\documentclass[12pt,a4paper]{scrartcl}
\usepackage[T1]{fontenc}
\usepackage[english]{babel}	
\usepackage{german}

\usepackage{graphicx}
\usepackage[a4paper]{geometry}
\usepackage{fullpage}
\usepackage{hyperref}

\usepackage{algorithm}
\usepackage{algpseudocode} 
\usepackage{amsmath}
\usepackage{amsthm}
\usepackage{amssymb}
\usepackage{mathpazo}
\usepackage{array}

\theoremstyle{definition}
\newtheorem{defi}{Definition}[section]
\newtheorem{rema}[defi]{Remark}

\newtheorem{conj}[defi]{Conjecture}

\theoremstyle{plain}

\newtheorem{prop}[defi]{Proposition}

\newtheorem{lemm}[defi]{Lemma}

\numberwithin{equation}{section}

\newcommand{\mN}{\mathbb{N}}
\newcommand{\mZ}{\mathbb{Z}}
\newcommand{\mP}{\mathbb{P}}
\newcommand{\mdot}{\!\cdot\!}
\newcommand{\ndiv}{\!\nmid\!}
\newcommand{\mdiv}{\!\mid\!}
\newcommand{\copr}{\!\perp\!}

\newcommand{\vl}{\vrule width 2pt}
\newcommand{\lb}{\linebreak}

\DeclareMathOperator{\modu}{mod} 

\usepackage{tikz}
\usetikzlibrary{shapes,arrows}
\tikzstyle{cloud} = [ellipse, draw, minimum height=1cm, minimum width=2cm]
\tikzstyle{block} = [rectangle, draw, minimum height=1cm, minimum width=3cm]
\tikzstyle{decision} = [diamond, draw, aspect=2, align=center, minimum height=1cm, minimum width=3cm, inner sep=-8pt]
\tikzstyle{term} = [rectangle, draw, rounded corners, minimum height=1cm, minimum width=3cm]
\tikzstyle{line} = [draw, -latex']

\usepackage{color}
\usepackage{colortbl}
\definecolor{lightgrey}{rgb}{0.85,0.875,0.85}   
\newcommand{\clg}{\cellcolor{lightgrey}}

\begin{document}

\title{New computational results\lb on a conjecture of Jacobsthal}
\author{Mario Ziller}
\date{}

\maketitle

\begin{abstract}

Jacobsthal's conjecture has been disproved by counterexample a few \lb years ago. We continue to verify this conjecture on a larger scale. For this purpose, we implemented an extension of the Greedy Permutation \lb Algorithm and computed the maximum Jacobsthal function for the \lb product of $k$ primes up to $k=43$.

We have found various new counterexamples. Their pattern seems to \lb imply that the conjecture of Jacobsthal only applies to several small $k$. \lb Our results raise further questions for discussion.

In addition to this paper, we provide exhaustive information about all covered sequences of the appropriate maximum lengths in ancillary files.\\
\end{abstract}

\section{Introduction}

\subsection*{\ \\Notation}

Henceforth, we denote the set of integral numbers by $\mZ$ and the set of natural \lb numbers, i.e. positive integers, by $\mN$. $\mP=\{p_i\mid i\in\mN\}$ is the ordered set of prime numbers with $p_1=2$. As usual, we define the $k^{th}$ primorial number as the product \lb of the first  $k$ primes:
$$p_k\#=\prod_{i=1}^k p_i\ , k\in\mN.$$

The number of prime divisors of a natural number $n\in\mN$ is the number of different primes which divide $n$:
$$d_\mP(n)=|\{p\in\mP \mid p/n\}|.$$\\

\subsection*{Jacobsthal's conjecture}

The Jacobsthal function $j(n)$, $n\in\mN$ is defined to be the smallest $m\in\mN$, such that every sequence of $m$ consecutive integers contains at least one element coprime to $n$ \cite{Jacobsthal_1960_I, Erdoes_1962}.

\begin{defi} {\itshape Jacobsthal function.}\\
For $n\in\mN$, the Jacobsthal function $j(n)$ is defined as
$$j(n)=\min\ \{m\in\mN\mid\forall\ a\in\mZ\ \exists\ x\in\{1,\dots,m\}:a+x\copr n\}.$$
\end{defi}

The entire Jacobsthal function is therefore determined by its values for products \lb of distinct primes \cite{Jacobsthal_1960_I}. In his subsequent elaborations \cite{Jacobsthal_1960_I, Jacobsthal_1960_II, Jacobsthal_1960_III, Jacobsthal_1961_IV, Jacobsthal_1961_V}, Jacobsthal derived explicit formulae for the calculation of function values for squarefree $n$ containing \lb up to seven distinct prime factors, and bounds on the function for up to ten distinct prime factors.\\

There are two derived functions related to sqarefree numbers. The particular \lb case of primorial numbers is the first step when investigating squarefree numbers. The Jacobsthal function applied to primorial numbers $h(k)$ is therefore defined as \lb the smallest $m\in\mN$, such that every sequence of $m$ consecutive integers contains an integer coprime to the product of the first $k$ primes.

\begin{defi} {\itshape Primorial Jacobsthal function.} \label{primorial_Jacobsthal}\\
For $k\in\mN$, the primorial Jacobsthal function $h(k)$ is defined as
$$h(k)=j(p_k\#).$$
\end{defi}

The more general case considers the product of $k$ arbitrary but different primes. Here, we initially ask for the maximum of the Jacobsthal function applied to any \lb natural number with $k$ prime factors. The reduction to squarefree numbers is \lb sufficient again.

\begin{defi} {\itshape Maximum Jacobsthal function.} \label{max_Jacobsthal}\\
Let $n,k\in\mN$. Then $H(k)$ is the maximum of the Jacobsthal function for products of $k$ different primes.
$$H(k) = \max_{d_\mP(n)=k}\ j(n) = \max_{\substack{\pi_i\in\mP \\ \pi_1<\dots<\pi_k}}\ j\!\left(\prod_{i=1}^k \pi_i\right).$$
\end{defi}

Obviously, $H(k)\ge h(k)$ follows by definition.\\

The computation of both functions is complicated and time-consuming. However, computational results for $h(k)$ are available for several years \cite{Hagedorn_2009, Ziller_Morack_2016}. Jacobsthal himself knew only very few function values. For small $k$, they were calculated without the help of computers. In all these cases, the results for both functions were the same. So, Jacobsthal assumed (\cite{Jacobsthal_1961_IV}, p. 3) that $H(k)=h(k)$ should hold for all $k\in\mN$.

\begin{conj} {\itshape Jacobsthal conjecture.} \label{Jacobsthal_conjecture}\\
For all $k\in\mN$,
$$H(k)=h(k).$$
\end{conj}

This assumption seems obvious at a first glance because less elements of a sequence are coprime to a smaller prime than to a larger one on average. Nevertheless, \lb Hajdu and Saradha \cite{Hajdu_Saradha_2012} found a counterexample: $H(24)\!>\!h(24)$. Moreover, they demonstrated that the conjecture holds for $k<24$.

The disproof of Jacobsthal's conjecture implies that there might be more cases \lb for which a particular choice of primes could lead to longer sequences without \lb coprimes than in case of only applying the smallest primes. In the following sections, we provide exhaustive computational results for the maximum Jacobsthal function $H(k)$ for all $k\le43$. In fact, there are quite a number of additional counterexamples. The detailed results rather suggest conjectures in the opposite direction.\\

\section{The computation of $H(k)$}

\ \\
The function $H(k)$ cannot be calculated by an unaltered utilisation of its definition because there are infinitely many choices of $k$ different primes. After introducing some divisibility properties of sequences of consecutive integers, we follow an idea of Hajdu and Saradha \cite{Hajdu_Saradha_2012} and first omit the sole even prime 2 from the calculation. This has proven sufficient and reduces unnecessary effort.

In a second step, we introduce an upper bound on the primes needed to be \lb considered \cite{Hajdu_Saradha_2012}. This problem reduction makes the computation of $H(k)$ possible. \lb Simultaneously, the validity of Jacobsthal's conjecture follows from it for $k\le19$.

Finally, we extend the recently developed Greedy Permutation Algorithm \cite{Ziller_Morack_2016} \lb and adapt it to the current problem. That approach ensures an efficient computation of $H(k)$ and an exhaustive search for all representative sequences of maximum length.\\

\subsection{Sequences and coverings}

A series of consecutive integers $\left(a+1,\dots ,a+m\right)$ where $a\in\mZ$ and $m\in\mN$ is shortly denoted by $\langle a\rangle_{m}$. We emphasise that $a$ itself is not member of $\langle a\rangle_{m}$.

A set of primes $\{\pi_i\}_{i=1}^k$ is called covering of a series of consecutive integers $\langle a\rangle_{m}$ \lb if every elements of the series is divisible by one of the given primes. We shortly say $\{\pi_i\}_{i=1}^k$ covers $\langle a\rangle_{m}$.

\begin{defi} {\itshape Covering.}\\
Let $a\in\mZ$, $m,k\in\mN$ and $\pi_i\in\mP,\ i=1,\dots,k$.
$$ \{\pi_i\}_{i=1}^k \text{ \,is a covering of } \langle a\rangle_{m}\ \iff\ \forall\ x\in\{1,\dots,m\}\ \exists\ i\in\{1,\dots,k\}:\pi_i\mdiv(a+x).$$
\end{defi}
\ 

The prime 2 plays a specific role in coverings. We formulate a lemma which makes the separation of prime 2 straightforward in future considerations. Extending a set of odd primes by 2 may more than double the length of the coverable sequence.

\begin{lemm} \label{two}\ \\
Let $m,k\in\mN$, $k\ge2$, and $\pi_i\in\mP,\ i=1,\dots,k$ with $\pi_1=2$.
$$ \exists\ a\in\mZ: \{\pi_i\}_{i=2}^k \text{ \,covers } \langle a\rangle_{m}\ \iff\ \exists\ b\in\mZ: \{\pi_i\}_{i=1}^k \text{ \,covers } \langle b\rangle_{2m+1}. $$
\end{lemm}

\begin{proof}
\ \\
($\Rightarrow$)
Given $\forall\ x\in\{1,\dots,m\}\ \exists\ i\in\{2,\dots,k\}:\pi_i\mdiv(a+x)$.

There exists $b\in\mZ$ with $b\equiv 2\mdot a\ (\modu \pi_i)$, $i=2,\dots,k$, and $b\equiv 1\ (\modu 2)$ due to the Chinese remainder theorem. Then
\begin{equation*} \begin{split}
b+2\mdot x&\equiv 2\mdot a+2\mdot x\ \ \,\,\equiv 2\mdot(a+x)\ (\modu \pi_i)\ \text{ for } \pi_i>2,\ x=1,\dots,m,\ \text{ and}\\
b+2\mdot x-1&\equiv 1+2\mdot x-1\equiv 0\qquad\quad\ \ \ (\modu 2)\ \ \ \text{ for } x=1,\dots,m+1.
\end{split} \end{equation*}
($\Leftarrow$)
Given $\forall\ x\in\{1,\dots,2\mdot m+1\}\ \exists\ i\in\{1,\dots,k\}:\pi_i\mdiv(b+x)$.

(1) Let $b\equiv 1\ (\modu 2)$. There exists $a\in\mZ$ with $2\mdot a\equiv b\ (\modu \pi_i)$, $i=2,\dots,k$. \lb For $x=1,\dots,m$, we get 
$$b+2\mdot x\equiv 2\mdot a+2\mdot x\equiv 2\mdot(a+x)\ (\modu \pi_i).$$

(2) Let $b\equiv 0\ (\modu 2)$. There exists $a\in\mZ$ with $2\mdot a+1\equiv b\ (\modu \pi_i)$, $i=2,\dots,k$. For $x=1,\dots,m$, we get 
$$b+2\mdot x-1\equiv 2\mdot a+2\mdot x\equiv 2\mdot(a+x)\ (\modu \pi_i).$$

In both cases, $a$ satisfies the requirements because $2\ndiv\pi_i$.
\end{proof}

\vspace{1mm}

\subsection{Restriction to odd primes} \label{odd}

Given $k$ different primes $\pi_i$, $i=1,\dots,k$, the function $H(k)$ by definition turns out \lb to be the smallest $m\in\mN$, such that every sequence of $m$ consecutive integers contains at least one element coprime to any of the $\pi_i$. This means on the other hand that $H(k)$ is the largest $m\in\mN$ for which a sequence of $m-1$ consecutive integers exists, such that each of its elements is divisible by one of the $\pi_i$.

For the determination of $H(k)$, it has proven sufficient to omit the smallest prime 2 from the calculation. Therefore, we define a function $\Omega(k)$ the computation of which reduces unnecessary effort. Subsequently, we will directly associate it with $H(k)$ in \lb a functional relationship.

We define $\Omega(k)$ as the maximum $m\in\mN$ for which a sequence of $m$ consecutive integers and $k-1$ different odd primes $\pi_i$, $i=2,\dots,k$ exist, such that each element \lb of the sequence is divisible by one of the $\pi_i$.

\begin{defi} {\itshape Maximum Jacobsthal function for odd numbers.} \label{Omega}\\
Let $m,k\in\mN$, $k\ge2$, and $\pi_i\in\mP>2,\ i=2,\dots,k$.
$$\Omega(k)=\max\ \{m\in\mN\mid\exists\ a\in\mZ\ \forall\ x\in\{1,\dots,m\}\ \exists\ i\in\{2,\dots,k\}:\pi_i/(a+x)\}.$$
\end{defi}

In other words, $\Omega(k)$ is the maximum $m\in\mN$ such that $a\in\mZ$ and $\pi_i\in\mP>2$, $i=2,\dots,k$ exist where $\{\pi_i\}_{i=2}^k$ covers $\langle a\rangle_{m}$.

\begin{rema}
We emphasise that the arguments $k$ of $H$ and $\Omega$ were harmonised for the case of $\pi_1=2$ is included in $H$. Furthermore, we remind that $H$ is completely determined by considering only squarefree numbers.
\end{rema}
\ 

For $n,k\in\mN$, the definition of $H(k)=\max_{d_\mP(n)=k} j(n)$ splits into two cases: $2\mdiv n$ \lb and $2\ndiv n$. Thus, we receive for the moment
$$H(k)=\max\{\max_{\substack{d_\mP(n)=k\\2\ \mdiv\ n}} j(n),\ \max_{\substack{d_\mP(n)=k\\2\ \ndiv\ n}} j(n)\}.$$
Considering $k$ arbitrary odd primes only,
$$\max_{\substack{d_\mP(n)=k\\2\ \ndiv\ n}} j(n) = \Omega(k+1)+1$$
follows by definition.\\

On the other hand, we can now relate $H$ and $\Omega$ for the case of coverings including the prime 2.

\begin{lemm} \ \\
Let $k\in\mN\ge2$.
$$\max_{\substack{d_\mP(n)=k\\2\ \mdiv\ n}} j(n) = 2\mdot\Omega(k)+2.$$
\end{lemm}

\begin{proof}
By definition \ref{Omega}, there exist $a\in\mZ$ and $\pi_i\in\mP>2,\ i=2,\dots,k$ such that $\{\pi_i\}_{i=2}^k$ covers $\langle a\rangle_{\Omega(k)}$. According to lemma \ref{two}, there exists $b\in\mZ$ for which $\{\pi_i\}_{i=1}^k$ covers $\langle b\rangle_{2\Omega(k)+1}$ when $\pi_1=2$. Maximality remains retained in both directions.
\end{proof}

For $k\in\mN\ge 2$, we finally conclude
$$H(k) = \max\{ \Omega(k+1)+1 ,\ 2\mdot\Omega(k)+2 \}.$$

For the determination of $H$-values, it is sufficient to focus on the computation of the function values of $\Omega$. We may limit to odd primes.\\

\subsection{Balanced coverings}

We note that $H(1)=2$  holds because for any prime $p$, $p\mdiv p$, and $p\ndiv(p\pm1)$. In case of $k=2$, we get $\Omega(k)=1$ for the same reason.
$$H(1)=2,\quad\quad \Omega(2)=1.$$

We now turn our attention to the case $k\ge3$ and go on to restrict the choice of primes. First, we demonstrate that we only need to consider coverings $\{\pi_i\}_{i=2}^k$ of a sequence of maximum length where any $\pi_i>p_k$ covers at least two positions in the sequence exclusively. We call such coverings balanced. In analogy with Hajdu and Saradha \cite{Hajdu_Saradha_2012}, we prove that the primes of balanced coverings can be bounded.

\begin{defi} {\itshape Balanced covering.}\\
Let $\{\pi_i\}_{i=2}^k$ be a covering of $\langle a\rangle_{m}$ with $a\in\mZ$, $m,k\in\mN$, $k\ge3$, and $\pi_i\in\mP>2$,\lb $\ i=2,\dots,k$.

$\{\pi_i\}_{i=2}^k$ is called a balanced covering of $\langle a\rangle_{m}$ if $m=\Omega(k)$ and for all $\pi_j>p_k$, $2\le j\le k$, there exist $x,y\in\mN$, $1\le x<y\le m$, with
\begin{equation*} \begin{split}
&\pi_j\mdiv(a+x) \land \pi_j\mdiv(a+y) \quad\text{and}\\
&\pi_i\ndiv(a+y) \land \pi_i\ndiv(a+y),\ i=2,\dots,k,\ i\ne j.
\end{split} \end{equation*}
\end{defi}

\begin{prop} \ \\
Let $a\in\mZ$, $k\in\mN\ge3$, and $\pi_i,{\pi_i}'\in\mP>2,\ i=2,\dots,k$.

For every covering $\{\pi_i\}_{i=2}^k$ of $\langle a\rangle_{\Omega(k)}$, there exist $b\in\mZ$ and a balanced covering $\{{\pi_i}'\}_{i=2}^k$\lb of $\langle b\rangle_{\Omega(k)}$.
\end{prop}

\begin{proof}
Let $\{\pi_i\}_{i=2}^k$ be a covering of $\langle a\rangle_{\Omega(k)}$.

Given any $j$, $2\le j\le k$ with $\pi_j>p_k,$ and a position $a+x$, $1\le x\le m$, exclusively covered by $\pi_j$, i.e. $\pi_j\mdiv(a+x)$ and $\pi_i\ndiv(a+x)$ for $i=2,\dots,k$, $i\ne j$. Such an $x\in\mN$ must exist because $\Omega(k)$ represents the maximum length. Otherwise, $\pi_j$ could be chosen \lb to cover position $a+\Omega(k)+1$.

Due to $\pi_j>p_k$, there is a $p_\nu\not\in\{\pi_i\}_{i=2}^k$ with $2\le\nu\le k$ and $p_\nu\le p_k$. If $a+x$ were the only position with the above property then we set $\{{\pi_i}'\}_{i=2}^k=\{\pi_i\}_{i=2}^k$ for $i\ne j$ and ${\pi_j}'=p_\nu$, and choose $b$ such that $b\equiv -x\ (\modu p_\nu)$ and $b\equiv a\ (\modu \pi_i)$ for $i=2,\dots,k$, $i\ne j$. Thus, $\{{\pi_i}'\}_{i=2}^k$ is a covering of $\langle b\rangle_{\Omega(k)}$with ${\pi_j}'<\pi_j$.
 
A non-balanced covering $\{\pi_i\}_{i=2}^k$ of $\langle a\rangle_{\Omega(k)}$ can therefore be reduced to a balanced covering by repeated application of the just described procedure.
\end{proof}

With that result, we can derive an upper bound on the primes necessary to consider for the computation of $\Omega(k)$.

\begin{prop} \label{bounded _primes}\ \\
Let $\{\pi_i\}_{i=2}^k$ be a balanced covering of $\langle a\rangle_{\Omega(k)}$ with $a\in\mZ$, $k\in\mN\ge3$, and $\pi_i\in\mP>2$,\lb $\ i=2,\dots,k$. Then $\pi_i\le q_k$ holds for all $i=2,\dots,k$ where
$$ q_k=\max\{p\in\mP \ |\ p\le  p_k \lor p\le \Omega(k-1)+1  \}.$$ 
\end{prop}

\begin{proof}
Let $\{\pi_i\}_{i=2}^k$ be a balanced covering of $\langle a\rangle_{\Omega(k)}$. Without loss of generality, \lb we assume $\pi_2\le\dots\le\pi_k$. Then $\pi_k\ge p_k$ is obvious.

For $\pi_k>q_k$, there could only be one position $x$, $1\le x\le m$ with $\pi_k\mdiv(a+x)$. \lb
We assume there were two positions $x,x+d\mdot\pi_k$ with $d\in\mN$, $1\le x<x+d\mdot\pi_k\le m$, \lb $\pi_k\mdiv(a+x)$, and consequently $\pi_k\mdiv(a+x+d\mdot\pi_k)$. Then $\{\pi_i\}_{i=2}^{k-1}$ must cover \lb $\langle a+x\rangle_{\pi_k-1}=(a\!\!\;+\!\!\;x\!\!\;+\!\!\;1,\dots,a\!\!\;+\!\!\;x\!\!\;+\!\!\;\pi_k\!\!\;-\!\!\;1)$. Here, we conclude $\pi_k-1 \le \Omega(k-1)$ \lb in contradiction with $\pi_k>q_k\ge \Omega(k-1)+1$.

Thus, the case $\pi_k>q_k$ violates the definition of balanced coverings.
\end{proof}

The previous proposition proves the restriction to balanced coverings sufficient. \lb For the computation of $\Omega(k)$, only odd primes $\pi_i\le q_k$ need to be included in the calculations. The scheme in figure \ref{scheme} depicts all necessary steps for the calculation \lb of $H(k)$. As demonstrated above, $\Omega(k)$ must always be known one step ahead.\\

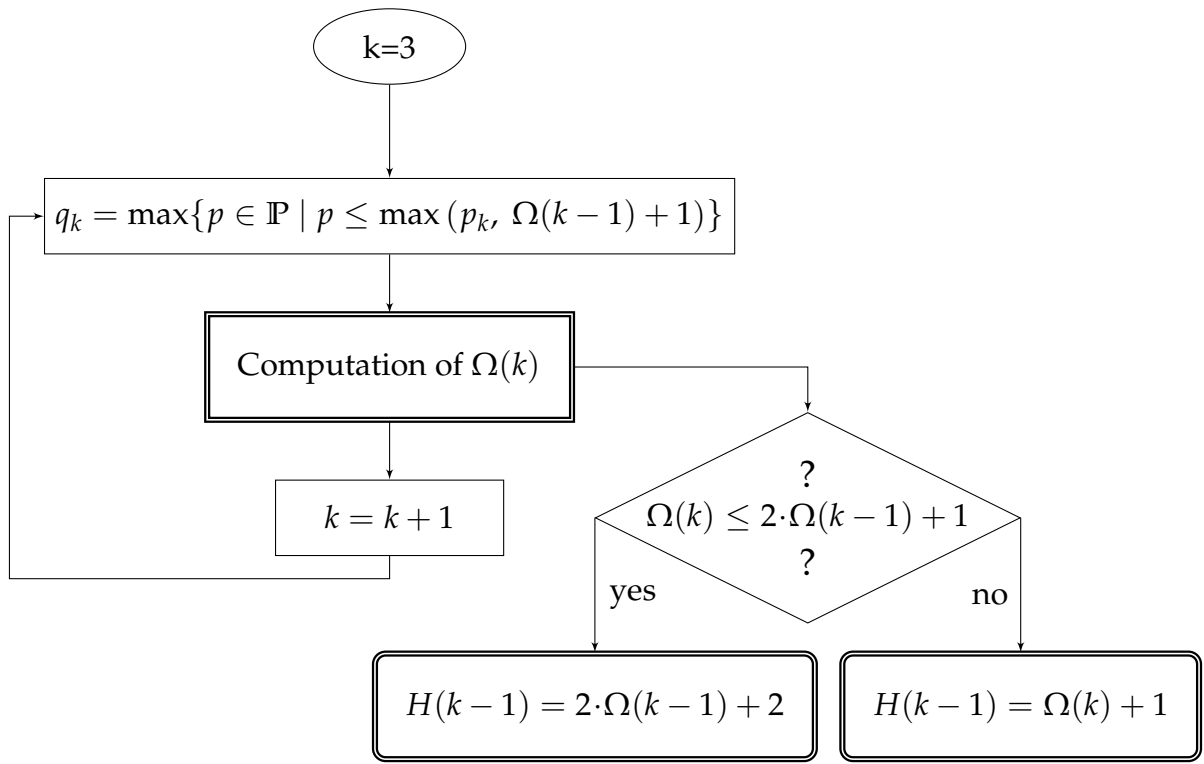
\begin{figure}[!h]
   \centering
\begin{tikzpicture}	
    \node [cloud] (init) {k=3};
    \node [block, below of=init, node distance=2.25cm] (q) {$q_k=\max\{p\in\mP \ |\ p\le \max \left( p_k ,\ \Omega(k-1)+1\right) \}$};
    \path [line] (init) -- (q);
    \node [block, below of=q, line width=0.3mm, double, minimum height=1.4cm, inner sep=4mm, node distance=2cm] (omega) {Computation of $\Omega(k)$};
    \path [line] (q) -- (omega);
    \node [block, below of=omega, node distance=2cm] (k) {$k=k+1$};
    \path [line] (omega) -- (k);
    \node [decision, right of=k, node distance=5.5cm] (decide) {\large{?}\\$\Omega(k)\le 2\mdot\Omega(k-1)+1$\vspace{0.3mm}\\\large{?}};
    \path [line] (omega.east) -| (decide.north);
    \node [term, left of=decide, below of=decide, line width=0.3mm, double, minimum height=1.4cm, inner sep=4mm, node distance=2.5cm, xshift=-3mm] (H1) {$H(k-1)=2\mdot\Omega(k-1)+2$};
    \path [line] (k.south) -- ++(0mm,-3mm)  -- ++(-5cm,0) |- (q.west);
    \path [line] (decide.west) -| node [xshift=5mm, yshift=-10.5mm] {yes} (H1);
    \node [term, right of=decide,  below of=decide, line width=0.3mm, double, minimum height=1.4cm, inner sep=4mm, node distance=2.5cm, xshift=3mm] (H2) {$H(k-1)=\Omega(k)+1$};
    \path [line] (decide.east) -| node [xshift=-4mm, yshift=-10.5mm] {no} (H2);
   \node [block, below of=H1, white, minimum height=0cm, inner sep=0mm, node distance=1cm] (blank) {};
\end{tikzpicture}
  \caption{Computation scheme.}
  \label{scheme}\ 
\end{figure}

We already know a wide range of values for the function $h(k)$ \cite{Hagedorn_2009, Ziller_Morack_2016}, which \lb represents the Jacobsthal function applied to primorial numbers. From the \lb computation of it, we also know values of $\omega(k)$ \cite{Hagedorn_2009, Ziller_Morack_2016} in addition. This function \lb is equivalent to $\Omega(k)$ for the first $k-1$ odd primes $\pi_i\le p_k$.

We note
\begin{equation*} \begin{split}
\omega(k-1)+1&<p_{k+1} \quad\quad\text{for } 3\le k\le19,\ \text{but}\\
76=\omega(19)+1&>p_{21}=73.
\end{split} \end{equation*}

Therefore, we have $q_k=p_k$ and the computation was always done for $3\le k\le19$. Furthermore, we get $q_{20}=p_{21}=73$ and conclude
$$\Omega(20)=\omega(21)=94\le 151=2\mdot\omega(19)+1=2\mdot\Omega(19)+1.$$

The remaining problem is the computation of $\Omega(k)$ for $k\ge21$. For the moment, \lb we can state that the conjecture of Jacobsthal holds at least for all $k\le19$. Together with the above mentioned results for $k\le2$, we obtain the following table \ref{tab19}.\\\\

\begin{table}[!h]
  \centering
\begin{tabular}{cc}	
\begin{tabular}{!\vl cc|cc !\vl}
  \noalign{\hrule height 2pt} 
  \rule{0pt}{14pt}$k$&$\!q_k\!=\!p_k\!$&$\!\Omega(k)\!=\!\omega(k)\!$&$\!H(k)\!=\!h(k)\!$\\[2pt]
  \noalign{\hrule height 2pt} 
  \rule{0pt}{14pt}\ \ 2&\,\ 3&\ \ 1&\ \ 4\\
\ \ 3&\,\ 5&\ \ 2&\ \ 6\\
\ \ 4&\,\ 7&\ \ 4&10\\
\ \ 5&11&\ \ 6&14\\
\ \ 6&13&10&22\\
\ \ 7&17&12&26\\
\ \ 8&19&16&34\\
\ \ 9&23&19&40\\
10&29&22&46\\
  [2pt]
  \noalign{\hrule height 2pt}
\end{tabular}
&
\begin{tabular}{!\vl cc|cc !\vl}
  \noalign{\hrule height 2pt} 
  \rule{0pt}{14pt}$k$&$\!q_k\!=\!p_k\!$&$\!\Omega(k)\!=\!\omega(k)\!$&$\!H(k)\!=\!h(k)\!$\\[2pt]
  \noalign{\hrule height 2pt} 
  \rule{0pt}{14pt}11&31&28&\ \ 58\\
12&37&32&\ \ 66\\
13&41&36&\ \ 74\\
14&43&44&\ \ 90\\
15&47&49&100\\
16&53&52&106\\
17&59&58&118\\
18&61&65&132\\
19&67&75&152\\
  [2pt]
  \noalign{\hrule height 2pt}
\end{tabular}
\end{tabular}
 \caption{Established results.}
 \label{tab19}
\end{table}

\subsection{An Algorithm for the computation of $\Omega(k)$}

The Greedy Permutation Algorithm (GPA) \cite{Ziller_Morack_2016} was developed for the computation of $\omega(k)$ which is equivalent to $\Omega(k)$ for the first $k-1$ odd primes $\pi_i\le p_k$. We extend this algorithm according to our current problem and thus make an efficient computation of $\Omega(k)$ possible. This approach also enables an exhaustive search for all corresponding balanced coverings.\\

The underlying idea of GPA is based on a specific order of choosing appropriate residue classes for each prime under consideration. This principle can contribute to the calculation of $\Omega(k)$ because any maximum covering can uniquely be represented by a set $\{a_i\}$ of residue classes to the corresponding modules $\pi_i$.

\pagebreak

\begin{prop} \label{residue_classes}\ \\
Let $\{\pi_i\}_{i=2}^k$ be a covering of $\langle a\rangle_{\Omega(k)}$ with $a\in\mZ$, $k\in\mN\ge3$, and $\pi_i\in\mP>2$,\lb $\ i=2,\dots,k$. Then there exist uniquely determined non-zero residue classes $a_i\modu\pi_i$ \lb such that every number of the sequence  $\langle 0\rangle_{\Omega(k)}$. belongs to one of them.
\begin{equation*} \begin{split}
&\exists\ a_i\in\{1,\dots,\pi_i-1\},\ i=2,\dots,k\\
&\forall\ x\in\{1,\dots,\Omega(k)\}\ \exists\ i\in\{2,\dots,k\}:x\equiv a_i\ (\modu\pi_i).
\end{split} \end{equation*}
\end{prop}

\begin{proof}
\ \\
$(\Rightarrow)$: \mbox{$a_i\equiv-a\ (\modu p_i)$}, \mbox{$i=2,\dots,k$} satisfy the respective congruences. $\pi_i\ndiv a$ because $\Omega(k)$ is maximum.\\
$(\Leftarrow)$: According to the Chinese remainder theorem, there exists an $a\in\mZ$ solving \lb the system of simultaneous congruences $a\equiv-a_i\ (\modu \pi_i),\ i=2,\dots,k$. With this solution, $\{\pi_i\}_{i=2}^k$ is a covering of $\langle a\rangle_{\Omega(k)}$ because for all respective $x$ there exists $i$ \lb with $x+a\equiv x-a_i\equiv 0\ (\modu \pi_i)$.
\end{proof}

GPA is a recursive algorithm. It starts with a given set of primes and an empty array representing the sequence $\langle a\rangle_m$ of a tentative length $m$. The algorithm tries to find suitable residuals $a_i\modu\pi_i$ such that as many positions of the sequence as possible can be covered. In each step, one of the remaining $a_i\modu\pi_i$ and therefore also the prime $\pi_i$ itself are chosen, the corresponding array elements are assigned, and the number of still free positions is compared with the maximum number of positions which can be covered by the remaining primes. If there is no more chance to fill the sequence then the current step will be skipped. If necessary, the length $m$ of the sequence will be enlarged until the maximum possible value is reached.

The number of conceivable combinations of residuals is enormous. The algorithm processes only a fraction of it. GPA chooses that residual first which covers most \lb of the free positions. It need not be reconsidered on the same recursion level again \lb because all possible combinations are checked with the first use. So, the maximum length of the sequence can be reached more quickly, and the rejection of bad \lb combinations can be done earlier. A detailed description of the Greedy Permutation Algorithm can be found in \cite{Ziller_Morack_2016}.\\

Some simple generalisations of GPA make the efficient computation of $\Omega(k)$ \lb possible. The set of primes which should be processed by the algorithm must be \lb extended to $p_2,\dots,p_k,\dots,q_k$. Moreover, the number of recursion steps must be \lb limited to $k-1$. This limit also has to be considered when counting the maximum number of coverable positions. And finally, every potential solution must be verified to be balanced.

The following pseudocode \ref{EGPA} summarises the algorithm applied in this paper.

\pagebreak

\begin{algorithm}[!h]
\caption{\ Extended Greedy Permutation Algorithm (EGPA).} \label{EGPA}
\begin{algorithmic} 
\Procedure{extended\_greedy\_permutation}{arr,i,ftab}
   \If {i<k-1}
      \If {i=1}
           \State fill\_frequency\_table\_of\_remainders(ftab)
           \State n\_empty=m	\Comment Starting number of free array positions
           \State n\_possible=count\_max\_possible\_covered\_positions(ftab)\\
			 				\Comment Starting number of maximum coverable positions
       \Else
         \State n\_empty=update\_free\_array\_positions(arr)
         \State n\_possible=update\_max\_possible\_covered\_positions(ftab)
      \EndIf

       \If {n\_possible$\ge$n\_empty}
          \State select\_appropriate\_$a_i$\_and\_$\pi_i$(ftab)
          \State arr1=arr; \ fill\_array(arr1,$a_i$,$\pi_i$)
          \State  ftab1=ftab; \ update\_frequency\_table\_of\_remainders(ftab1)
          \State extended\_greedy\_permutation(arr1,i+1,ftab1)	\Comment Permutation level i+1

          \State delete\_frequency\_of\_$a_i$\_mod\_$\pi_i$(ftab)
           \State extended\_greedy\_permutation(arr,i,ftab)	\Comment Permutation level i
      \EndIf
   \Else\ count\_array(arr)
      \If {longer\_sequence\_found}
         \State increase\_m
         \State update\_data\_structures
      \Else
         \If {sequence\_is\_covered and covering\_is\_balanced}
             \State record\_covering
         \EndIf
      \EndIf
   \EndIf
\EndProcedure

\vspace{1.75mm} \hrule \vspace{1.75mm}

\State m=starting\_sequence\_length	\Comment Starting sequence length
\State arr=empty\_array		\Comment Sequence array
\State plist=[$p_2,\dots,p_k,\dots,q_k$]	\Comment Array of primes
\State i=1	\Comment Starting prime array index
\State ftab=empty\_table	\Comment Frequency table of remainders
\State extended\_greedy\_permutation(arr,i,ftab)	\Comment Recursion
\end{algorithmic}
\end{algorithm}

\pagebreak

\section{Results}

\ \\
The Extended Greedy Permutation Algorithm \ref{EGPA} was successfully applied to the \lb computation of $\Omega(k)$ for $20\le k\le43$. The outcome was used to calculate the \lb corresponding $H(k)$-values as depicted in figure \ref{scheme}. $\Omega(44)<320$ was verified for the assessment of $H(43)$ as well. As an additional result, we received all existing balanced coverings of length $\Omega(k)$ in that range and provide this data in ancillary files.\\

\subsection{Summarised results}

In accord with Hajdu and Saradha \cite{Hajdu_Saradha_2012}, we can confirm Jacobsthal's conjecture \ref{Jacobsthal_conjecture} $H(k)=h(k)$ for $20\le k\le23$ whereas $H(24)>h(24)$. Thus, we ascertained $k=24$ to be the smallest counterexample. In our further calculations we found that $H(k)=h(k)$ holds for $k= 25,26,28,29,31,32$ only. In all other cases, Jacobsthal's conjecture was violated. Moreover, the difference $H(k)-h(k)$ seems to grow on average.\\

Table \ref{tab_new} summarises the findings including\vspace{-2mm}
\begin{itemize}
\setlength{\itemsep}{-1mm}
\item the index $k$ for $k=20,\dots,43$,
\item the $k^{th}$ prime $p_k$,
\item the maximum processed prime $q_k$ according to proposition \ref{bounded _primes},
\item the primorial Jacobsthal function $h(k)$ according to definition \ref{primorial_Jacobsthal},
\item the maximum Jacobsthal function $H(k)$ according to definition \ref{max_Jacobsthal},
\item the function $\Omega(k)$ according to definition \ref{Omega}, and
\item the number $n_{cov}$ of balanced coverings containing at least one prime $>p_k$.
\end{itemize}\vspace{-2mm}
Violations of Jacobsthal's conjecture are highlighted.\\

We emphasise that $n_{cov}$ counts only balanced coverings of the sequences of length $\Omega(k)$ containing at least one prime $\pi_i>p_k$.

In case of $H(k)=h(k)$, there are always coverings with $\pi_i\le p_k$ for $i=2,\dots,k$ which were completely provided in \cite{Ziller_Morack_2016}. Additional coverings counted by $n_{cov}$ could be found for $k=21,22,26,29$ only, and none for  $k=20,23,25,28,31,32$.

In case of $H(k)>h(k)$, other balanced coverings cannot exist. At least one prime $\pi_i>p_k$ must be included by definition.

\pagebreak

\begin{table}[!h]
  \centering
  \setlength{\tabcolsep}{4mm} 
\begin{tabular}{!\vl r|rr|rr|rr !\vl}
  \noalign{\hrule height 2pt} 
\rule{0pt}{14pt}$k$&$p_k$&$q_k$&h(k)&H(k)&$\Omega(k)$&$n_{cov}$\\[2pt]
  \noalign{\hrule height 2pt} 
\rule{0pt}{14pt}20&71&73&174&174&86&0\\
21&73&83&190&190&94&48\\
22&79&89&200&200&99&180\\
23&83&97&216&216&107&0\\
24&89&107&\clg234&\clg236&117&12\\
25&97&113&258&258&128&0\\
26&101&127&264&264&131&320\\
27&103&131&\clg282&\clg284&141&216\\
28&107&139&300&300&149&0\\
29&109&149&312&312&155&2074\\
30&113&151&\clg330&\clg332&165&48\\
31&127&163&354&354&176&0\\
32&131&173&378&378&188&0\\
33&137&181&\clg388&\clg390&194&14\\
34&139&193&\clg414&\clg420&209&4\\
35&149&199&\clg432&\clg438&218&8\\
36&151&211&\clg450&\clg462&230&2\\
37&157&229&\clg476&\clg482&240&4\\
38&163&241&\clg492&\clg500&249&2\\
39&167&241&\clg510&\clg520&259&116\\
40&173&257&\clg538&\clg544&271&4\\
41&179&271&\clg550&\clg566&282&4\\
42&181&283&\clg574&\clg588&293&4\\
43&191&293&\clg600&\clg610&304&2\\
 [2pt]
  \noalign{\hrule height 2pt}
\end{tabular}
 \caption{Computational results.}
 \label{tab_new}
\end{table}

\subsection{Ancillary data}

In addition to this paper, we provide three ancillary files each of which presents the exhaustive results of our
calculations in a specific format. These files include all $n_{cov}$ balanced coverings of the respective sequences of length $\Omega(k)$ containing at least one prime larger than $p_k$.

In the cases in which Jacobthal's conjecture \ref{Jacobsthal_conjecture} applies, there exist further \lb relevant coverings which however contain only primes not exceeding $p_k$. For detailed information about these coverings we refer to \cite{Ziller_Morack_2016}.\\

A sequence completely covered by a set of primes can uniquely be represented \lb in three ways. Detailed explanatory notes on this topic can be found in \cite{Ziller_Morack_2016}. For \lb each of these representations, we provide an analogous separated file to facilitate \lb comparisons with previous results.\\

\textbf{moduli\_c.txt}\\

In this file, every position $x\in\{1,\dots,m\}$ of a sequence is characterised by the \lb smallest prime $p_j$,  $j\in\{2,\dots,k\}$ covering $x$.\\

\textbf{remainders\_c.txt}\\

This file includes the ordered set of remainders $a_i\modu\pi_i$, $i\in\{2,\dots,k\}$ for every sequence as described in proposition \ref{residue_classes}.\\

\textbf{permutations\_c.txt}\\

The third file represents every covering by a uniquely determined permutation \lb of the primes under consideration. Starting with a prime which covers position 1, the permutation order follows from the coverage of the next still uncovered position at each time until the sequence is completely covered.\\

The sequences in \dq remainders\_c.txt\dq\ are separately sorted  for each $k$ by ascending\lb remainders. This order was maintained in the other files \dq permutations\_c.txt\dq\ and\lb \dq moduli\_c.txt\dq\ to make a direct comparison possible.

\subsection{Final remarks}

The conjecture \ref{Jacobsthal_conjecture} of Jacobsthal has previously been disproved. This paper provides computational results extending far beyond. New questions arise with it. We dare make the following three conjectures and put them up for discussion.

\begin{conj}
$$H(k)>h(k) \quad\text{for all}\ k\ge33.$$
\end{conj}

With the detection of the first counterexample, Hajdu and Saradha \cite{Hajdu_Saradha_2012} asked for whether there are more or even infinitely many counterexamples or not. Our results, \lb as shown in table \ref{tab_new}, rather imply questions in the opposite direction: Are there \lb infinitely many $k$ satisfying Jacobsthal's conjecture? Or might $k=32$ be the last one?

\begin{conj}
$$H(k)<2\mdot H(k-1) \quad\text{for all}\ k\ge3.$$
\end{conj}

The specified assertion is equivalent to $\Omega(k)\le2\mdot\Omega(k-1)+1$ for all $k\ge3$. \lb Its validity would simplify the calculation of $H(k)$ as outlined in subsection \ref{odd}. The assumption was already made in \cite{Hajdu_Saradha_2012}. All data known so far support the conjecture.

A simple upper bound could be derived as a corollary of the conjecture: $H(k)<2^k$ for $k\ge3$. This bound is in fact not tight. However, the weaker version $H(k)\le2^k$ has always been proven by Kanold \cite{Kanold_1967} by other means.

\begin{conj}
$$H(k)<k^2 \quad\text{for all}\ k\ge3.$$
\end{conj}

This would be a fairly strong upper bound. Kanold proved $H(k)\le k^2$ for the cases $2\le k\le12$ \cite{Kanold_1967} without knowing the latest function values.

An elementary proof of the assumption could directly be continued to an \lb elementary proof of Dirichlet's theorem \cite{Mercer_2017}. Furthermore, the theorem of Linnik would also be a conclusion from this conjecture according to \cite{Kanold_1965}.

All computed data satisfy the assumption.

\subsubsection*{Contact} 
marioziller@arcor.de

\bibliography{References}     

\begin{thebibliography}{10}

\bibitem{Erdoes_1962}
Paul Erd\"os, \emph{On the {Integers} {Relatively} {Prime} to $n$ and a
  {Number}-{Theoretic} {Function} {Considered} by {Jacobsthal}}, MATHEMATICA
  SCANDINAVICA \textbf{10} (1962), 163--170.

\bibitem{Hagedorn_2009}
Thomas~R. Hagedorn, \emph{Computation of {Jacobsthal}'s {Function} h(n) for
  n\;<\;50}, Mathematics of Computation \textbf{78} (2009), no.~266,
  1073--1087.

\bibitem{Hajdu_Saradha_2012}
L.~Hajdu and N.~Saradha, \emph{Disproof of a conjecture of Jacobsthal},
  Mathematics of Computation \textbf{81} (2012), no.~280, 2461--2471.

\bibitem{Jacobsthal_1960_I}
Ernst Jacobsthal, \emph{{\"Uber Sequenzen ganzer Zahlen, von denen keine zu n
  teilerfremd ist. I}}, D.K.N.V.S. Forhandlinger \textbf{33} (1960), no.~24,
  117--124.

\bibitem{Jacobsthal_1960_II}
Ernst Jacobsthal, \emph{{\"Uber Sequenzen ganzer Zahlen, von denen keine zu n
  teilerfremd ist. II}}, D.K.N.V.S. Forhandlinger \textbf{33} (1960), no.~25,
  125--131.

\bibitem{Jacobsthal_1960_III}
Ernst Jacobsthal, \emph{{\"Uber Sequenzen ganzer Zahlen, von denen keine zu n
  teilerfremd ist. III}}, D.K.N.V.S. Forhandlinger \textbf{33} (1960), no.~26,
  132--139.

\bibitem{Jacobsthal_1961_IV}
Ernst Jacobsthal, \emph{{\"Uber Sequenzen ganzer Zahlen, von denen keine zu n
  teilerfremd ist. IV}}, D.K.N.V.S. Forhandlinger \textbf{34} (1961), no.~1,
  1--7.

\bibitem{Jacobsthal_1961_V}
Ernst Jacobsthal, \emph{{\"Uber Sequenzen ganzer Zahlen, von denen keine zu n
  teilerfremd ist. V}}, D.K.N.V.S. Forhandlinger \textbf{34} (1961), no.~24,
  110--115.

\bibitem{Kanold_1965}
Hans-Joachim Kanold, \emph{{\"U}ber Primzahlen in arithmetischen Folgen. II}, Mathematische Annalen \textbf{157} (1965), no.~5, 358--362.

\bibitem{Kanold_1967}
Hans-Joachim Kanold, \emph{{\"U}ber eine zahlentheoretische Funktion von Jacobsthal}, Mathematische Annalen \textbf{170} (1967), no.~4, 314--326.

\bibitem{Mercer_2017}
Idris Mercer, \emph{Dirichlet's theorem and Jacobsthal's function}, arXiv:1708.05415 [math.NT] (2017).

\bibitem{Ziller_Morack_2016}
Mario Ziller and John~F. Morack, \emph{Algorithmic concepts for the computation of {Jacobsthal}'s function}, arXiv:1611.03310 [math.NT] (2016).

\end{thebibliography}

\end{document}